\documentclass[english,12pt]{amsart}
\usepackage[english]{babel}
\usepackage{amsmath,amssymb,enumerate,amsthm}
\usepackage[latin1]{inputenc}
\usepackage{cite}
\usepackage{latexsym}
\usepackage{setspace}

\makeatletter
\@namedef{subjclassname@2010}{%
  \textup{2010} Mathematics Subject Classification}
\makeatother

\newtheorem{theorem}{Theorem}[section]
\newtheorem{prop}[theorem]{Proposition}

\newtheorem{lemma}[theorem]{Lemma}

\theoremstyle{definition}
\newtheorem{defn}[theorem]{Definition}

\newcommand{\Bad}{\mathrm{Bad}}
\newcommand{\Tad}{\mathrm{Bad}_{\theta}(j_1,\ldots,j_n)}

\newcommand{\mL}{\mathcal{L}}

\def\={\;=\;}
\def\>{\;>\;}
\def\<{\;<\;}
\def\:{\,:\,}
\def\.={\;\dot{=}\;}

\newcommand{\R}{\mathbb{R}}

\newcommand{\Z}{\mathbb{Z}}
\newcommand{\N}{\mathbb{N}}

\newcommand{\V}{\mathcal{V}}

\newcommand{\M}{\mathcal{M}}
\newcommand{\Mk}{\mathcal{M}_k}
\newcommand{\Qk}{\mathcal{Q}_k}
\newcommand{\G}{\mathcal{G}}
\newcommand{\F}{\mathcal{F}}

\begin{document}

\baselineskip=17pt

\title[Twisted badly approximable points] {Badly approximable points in twisted Diophantine approximation and Hausdorff dimension}

\author[P. Bengoechea]{Paloma BENGOECHEA$^\dag$}
\thanks{$^\dag$ Research supported by EPSRC Programme Grant: EP/J018260/1}
\address{Department of Mathematics, University of York, York, YO10 5DD, United Kingdom}
\email{paloma.bengoechea@york.ac.uk}

\author[N. Moshchevitin]{Nikolay MOSHCHEVITIN$^*$}
\thanks{$^*$ Research supported by RFBR grant No. 15-01-05700a}
\address{Department of Mathematics and Mechanics, Moscow State University, Leninskie Gory 1, GZ MGU, 119991 Moscow, Russia}
\email{moshchevitin@rambler.ru}

\date{}

\begin{abstract}For any $j_1,\ldots,j_n> 0$ with $\sum_{i=1}^nj_i=1$ and any $\theta\in\R^n$, let ${\mathrm{Bad}_{\theta}(j_1,\ldots,j_n)}$ denote the set of points $\eta\in\R^n$ for which $\max_{1\leq i\leq n}(\|q\theta_i-\eta_i\|^{1/j_i})>c/q$ for some positive constant $c=c(\eta)$ and all $q\in\N$. These sets are the `twisted' inhomogeneous analogue of $\mathrm{Bad}(j_1,\ldots,j_n)$ in the theory of simultaneous Diophantine approximation. It has been shown that they have full Hausdorff dimension in the non-weighted setting, i.e provided that $j_i=1/n$, and in the weighted setting when $\theta$ is chosen from $\mathrm{Bad}(j_1,\ldots,j_n)$. We generalise these results proving the full Hausdorff dimension in the weighted setting without any condition on $\theta$. Moreover, we prove $\dim(\Tad\cap\Bad(1,0,\ldots,0)\cap\ldots\cap\Bad(0,\ldots,0,1))=n$.
\end{abstract}



\maketitle

\section{Introduction}

The classical result due to Dirichlet: for any real number $\theta$ there exist infinitely many natural numbers $q$ such that
\begin{equation}\label{Dirichlet}
\|q\theta\|\leq q^{-1},
\end{equation}
where $\|\cdot\|$ denotes the distance to the nearest integer, has higher dimension generalisations. Consider any $n$-tuple of real numbers $(j_1,\ldots,j_n)$ such that
\begin{equation}\label{s,t}
j_1,\ldots,j_n> 0\quad\mbox{and}\quad \sum_{i=1}^n j_i=1.
\end{equation}
Then, for any vector $\theta=(\theta_1,\ldots,\theta_n)\in\R^n$, there exist infinitely many natural numbers $q$ such that
\begin{equation}\label{bua}
\max_{1\leq i\leq n}(\|q\theta_i\|^{1/{j_i}})\leq q^{-1}.
\end{equation}
The two results above motivate the study of real numbers and real vectors $\theta\in\R^n$ for which the right hand side of \eqref{Dirichlet} and \eqref{bua} respectively cannot be improved by an arbitrary constant. They respectively constitute the sets $\Bad$ of badly approximable numbers and $\Bad(j_1,\ldots,j_n)$ of $(j_1,\ldots,j_n)$-badly approximable numbers. 
Hence
$$
\Bad(j_1,\ldots,j_n):=\left\{(\theta_1,\ldots,\theta_n)\in\R^n\, :\, \inf_{q\in\N}\, \max_{1\leq i\leq n}(q^{j_i}\|q\theta_i\|)>0\right\}.
$$
In the 1-dimensional case, it is well known that the set of badly approximable numbers has Lebesgue measure zero but maximal Hausdorff dimension. In the $n$-dimensional case, it is also a classical result that $\Bad(j_1,\ldots,j_n)$ has Lebesgue measure zero, and Schmidt proved in 1966  that the particular set $\Bad(1/2,1/2)$ has full Hausdorff dimension. But the result of maximal dimension in the weigthed setting hasn't been proved until almost 40 years later, by Pollington and Velani \cite{PV}. In the 2-dimensional case, An showed in \cite{A} that $\Bad(j_1,j_2)$ is in fact winning for the now famous Schmidt games -see \cite{Schmidt}. Thus he provided a direct proof of a conjecture of Schmidt stating that any countable intersection of sets $\Bad(j_1,j_2)$ is non empty -see also \cite{BPV}. 

Recently, interest in the size of related sets, usually referred to as the `twists' of the sets $\Bad(j_1,\ldots,j_n)$, has developed. The study of these new sets started in the 1-dimensional setting:
we fix $\theta\in\R$  and consider the twist of $\Bad$:
$$
\Bad_\theta:=\left\{\eta\in\R:\, \inf_{q\in\N}q\|q\theta-\eta\|>0\right\}.
$$
The set $\Bad_\theta$ has a palpable interpretation in terms of rotations of the unit circle. Identifying the circle with the unit interval $[0,1)$, the value $q\theta$ (modulo 1) may be thought of as the position of the origin after $q$ rotations by the angle $\theta$.
If $\theta$ is rational, the rotation is periodic. If $\theta$ is irrational, a classical result of Weyl \cite{W} implies that $q\theta$ (modulo 1) is equidistributed, so $q\theta$ visits any fixed subinterval of $[0,1)$ infinitely often. 
The natural question of what happens if the subinterval is allowed to shrink with time arises. Shrinking a subinterval corresponds to making its length decay according to some specified function. The set $\Bad_\theta$ corresponds to considering, for any $\epsilon>0$, the shrinking interval $(\eta-\epsilon/q,\eta+\epsilon/q)$ centred at the point $\eta$ and where the specified function is $\epsilon/q$. Khintchine showed in \cite{Kin} that
\begin{equation}\label{Kin}
\|q\theta-\eta\|<\dfrac{1+\delta}{\sqrt{5}q}\qquad(\delta>0)
\end{equation}
is satisfied for infinitely many integers $q$, and Theorem III in Chapter III of Cassels' book \cite{Cassels} shows that the right hand side of \eqref{Kin} cannot be improved by an arbitrary constant for every irrational $\theta$ and every real $\eta$. This motivates the study of the set $\Bad_\theta$. Kim \cite{K} proved in 2007 that it has Lebesgue measure zero, and later it was shown by Tseng \cite{T} that it has full Hausdorff dimension (actually Tseng proved that $\Bad_\theta$ has the stronger property of being winning for any $\theta\in\R$).

By generalising circle rotations to rotations on torus of higher dimensions, i.e. by considering the sequence $q\theta$ (modulo 1) in $[0,1)^n$ where $\theta=(\theta_1,\ldots,\theta_n)\in\R^n$, we obtain the `twists' of the sets $\Bad(j_1,\ldots,j_n)$:
\begin{equation}
\Tad=\left\{(\eta_1,\ldots,\eta_n)\in\R^n:\, \inf_{q\in\N}\, \max_{1\leq i\leq n}(q^{j_i}\|q\theta_i-\eta_i\|)>0\right\}.
\end{equation}

In \cite{B} Bugeaud et al proved that the non-weighted set $\Bad_\theta(1/n,\ldots\\\ldots,1/n)$ has full Hausdorff dimension. Recently, Einsiedler and Tseng \cite{ET} extended the results \cite{B} and \cite{T} by showing, among other results, that $\Bad_\theta(1/n,\ldots,1/n)$ is also winning. 
It was shown in \cite{Mr1} that such results may be obtained by classical methods developed by
Khintchine \cite{Kin1} and Jarn\'{\i}k \cite{Jr1,Jr2} and discussed in Chapter V of Cassels' book \cite{Cassels}.
Unfortunately, these methods cannot be directly extended to the weighted setting. For the weighted setting, less has heretofore been known.
Harrap did the first contribution \cite{H} in the 2-dimensional case, by proving that $\Bad_\theta(j_1,j_2)$ has full Hausdorff dimension provided that the fixed point $\theta\in\R^2$ belongs to $\Bad(j_1,j_2)$, which is a significantly restrictive condition.
Recently,  under the hypothesis $\theta\in\Bad(j_1,\ldots,j_n)$, Harrap and Moshchevitin have extended to weighted linear forms in higher dimension and improved to winning the result in \cite{H} (see \cite{HM}).

In this paper, we prove that the weighted set $\Tad$ has full Hausdorff dimension for any $\theta\in\R^n$. Moreover, the following theorem holds. 

\begin{theorem}\label{teorema}
For any $\theta\in\R^n$ and all $j_1,\ldots,j_n> 0$ with $\sum_{i=1}^nj_i=1$,
$$
\dim(\Tad\cap\mathrm{Bad}(1,0,\ldots,0)\cap\ldots\cap\mathrm{Bad}(0,\ldots,0,1))=n.
$$
\end{theorem}

The same type of theorem holds in the classical not twisted setting; it constitutes the work done in \cite{PV} (see Theorem 2).
\\

Note that if $1,\theta_1,\ldots,\theta_n$ are linearly dependent over $\Z$, then Theorem \ref{teorema} is obvious. Indeed, in this case $\left\{q\theta:\, q\in\Z\right\}$ is restricted to  a hyperplane $H$ of $\R^n$, so $\Tad\supset\R^n\backslash H$ is winning. Hence $\Tad\cap\mathrm{Bad}(1,0,\ldots,0)\cap\ldots\cap\mathrm{Bad}(0,\ldots,0,1)$ is winning and in particular has full dimension \footnote{We recall that winning sets in $\R^n$ have maximal Hausdorff dimension, and that countable intersections of winning sets are again winning. We refer the reader to \cite{Schmidt} for all necessary definitions and results on winning sets.}. Therefore we suppose throughout the paper that $1,\theta_1,\ldots,\theta_n$ are linearly independent over $\Z$.

The strategy for the proof of Theorem \ref{teorema} is as follows. We start by defining a set $\V\subset\Tad$ related to the best approximations to the fixed point $\theta\in\R^n$. 
Then we construct a Cantor-type set $K(R)$ inside $\V\cap\Bad(1,0,\ldots,0)\cap\ldots\cap\Bad(0,\ldots,0,1)$. Finally we describe a probability measure supported on $K(R)$ to which we can apply the mass distribution principle and thus find a lower bound for
 the dimension of $K(R)$. 

Best approximations are defined in Section 2. In Section 3 we define $\V$ and give the proof of the inclusion $\V\subset\Tad$. We construct $K(R)$ in Section 4 and describe the probability measure in Section 5. Finally we compute the lower bound for the dimension of $K(R)$ in Section 6.
\\

In the following, we let $n\in\N$, fix an $n$-tuple $(j_1,\ldots,j_n)\in\R^n$ satisfying \eqref{s,t} and a vector $\theta=(\theta_1,\ldots,\theta_n)\in\R^n$ such that $1,\theta_1,\ldots,\theta_n$ are linearly independent over $\Z$. We denote by $x\cdot y$ the scalar product of two vectors $x$ and $y$ in $\R^n$, and by $\|\cdot\|$ the distance to the nearest integer.

\section{Best approximations}

\begin{defn}\label{def ba}
An n-dimensional vector $m=(m_1,\ldots,m_n)\in\Z^n\backslash\left\{0\right\}$ is called a best approximation to $\theta$ if for all $v\in\Z^n\backslash\left\{0,-m,m\right\}$ the following implication holds:
$$
\max_{1\leq i\leq n}(|v_i|^{1/j_i})\leq\max_{1\leq i\leq n}(|m_i|^{1/j_i})\Longrightarrow 
\|v\cdot\theta\|>\|m\cdot\theta\|.
$$
\end{defn}

Note that the condition $1,\theta_1,\ldots,\theta_n$ are $\Z$-linearly independent allows us to demand a strict inequality in the right hand side of the implication above.

Note also that when $n=1$ the best approximations to a real number $x$ are, up to the sign, the denominators of the convergents to $x$.

Since $1,\theta_1,\ldots, \theta_n$ are $\Z$-linearly independent, we have an infinite number of best approximations to $\theta$. They can be arranged up to the sign -so that two vectors of opposite sign do not both appear- in an infinite sequence
\begin{equation}\label{sequence}                                                                        
m_\nu=(m_{\nu,1},\ldots,m_{\nu,n})\qquad\nu\geq 1,
\end{equation}
such that the values
\begin{equation}\label{M}
M_\nu=\max_{1\leq i\leq n}(|m_{\nu,i}|^{1/j_i})
\end{equation}
form a strictly increasing sequence, and the values 
\begin{equation}\label{zeta}
\zeta_\nu=\|m_\nu\cdot\theta\|
\end{equation}                                                                    
form a strictly decreasing sequence. Hence each value $M_\nu$ corresponds to a single best approximation $m_\nu$. The quantity $M_\nu$ can be referred to as the `height' of $m_\nu$.

Best approximations vectors have often been used in proofs, but not always explicitly. In particular, Voronoi  \cite{Vor} selected some points in
a lattice that correspond exactly to the best approximation vectors (see  also \cite{DelFad}).  Similar constructions were introduced in \cite{L} or Section 2 of \cite{BL}. Some important properties of the best approximation vectors are discussed in \cite{Mba,Mhin} and a recent survey on the topic is due to Chevallier \cite{Nicolas}. 
\\

For each $\nu\geq 1$, it is easy to see that the region 
$$
\left\{(x_0,\ldots,x_n)\in\R^{n+1}\, :\, \max_{1\leq i\leq n}(|x_i|^{1/j_i})<M_{\nu+1},\,\, \Big|x_0+\sum_{i=1}^n x_i\theta_i\Big|<\zeta_\nu\right\}
$$
does not contain any integer point different from $0$. Since this region has volume $2^{n+1}M_{\nu+1}\zeta_\nu$ (see Lemma 4 in Appendix B of \cite{Cassels}), it follows from Minkowski's convex body theorem that
\begin{equation}\label{Minko}
\zeta_\nu M_{\nu+1}\leq 1.
\end{equation}
The inequality above will be used later as well as the following lemma, stating that the sequence of heights $M_\nu$ is lacunary.

\begin{lemma}\label{lema lacunary}
For every $\nu\geq 1$, we have
$$
M_{\nu+2\cdot 3^n}\geq 2M_{\nu}.
$$
\end{lemma}

\begin{proof} Given $\nu\geq 1$, we show that we have at most $2\cdot 3^n$ vectors $m_{\nu+r}$ with $r\geq 0$ and $M_{\nu+r}< 2M_\nu$.
The goal is to see that the $0$-symmetric region
\begin{equation}\label{double}
\left\{(x_0,\ldots,x_n)\in\R^{n+1}:\, \max_{1\leq i\leq n}(|x_i|^{1/j_i})< 2M_\nu,\,\, \Big|x_0+\sum_{i=1}^n x_i\theta_i\Big|\leq\zeta_\nu\right\}
\end{equation} 
contains at most $4\cdot3^n$ integer points other than $0$. The region \eqref{double} is covered by sets of the form 
$$
T(\xi)=\left\{\begin{array}{rl}(x_0,\ldots,x_n)\in\R^{n+1}&:\, \max_{1\leq i\leq n}(|x_i-\xi_i|^{1/j_i})\leq M_\nu,\\
 &\mbox{and } \Big|x_0-\xi_0+\sum_{i=1}^n(x_i-\xi_i)\theta_i\Big|\leq\zeta_\nu\end{array}\right\},
$$
with
\begin{equation}\label{xi}
\xi_i\in\left\{-2M_\nu^{j_i},0,2M_\nu^{j_i}\right\}, \qquad \xi_0=-\sum_{i=1}^n\xi_i\theta_i.
\end{equation}
Each region $T(\xi)$ is the translate by $(\xi_0,\ldots,\xi_n)$ of the set 
$$
\left\{(x_0,\ldots,x_n)\in\R^{n+1}:\, \max_{1\leq i\leq n}(|x_i|^{1/j_i})\leq M_\nu,\,\, \Big|x_0+\sum_{i=1}^n x_i\theta_i\Big|\leq\zeta_\nu\right\},
$$
which contains exactly three integer points: $0$ and two best approximations with opposite sign. Hence each $T(\xi)$ contains at most four integer points. Since there are $3^n$ possible choices for $(\xi_0,\ldots,\xi_n)$ satisfying \eqref{xi},
 the set \eqref{double} contains at most $4\cdot 3^n$ integer points.
\end{proof}

\section{The set $\V$ included in $\Tad$}

The following proposition allows us to work with a set defined by the best approximations to $\theta$ instead of working directly with $\Tad$. 

\begin{prop}\label{prop} If $\eta\in\R^n$ satisfies
\begin{equation}\label{Argu Mosche}
\inf_{\nu}\|m_\nu\cdot\eta\|>0,
\end{equation}
then $\eta\in\Tad$.
\end{prop}

\begin{proof} Let $\eta=(\eta_1,\ldots,\eta_n)\in\R^n$ satisfy
$$
\| m_\nu\cdot\eta\|>\gamma\qquad \forall\nu\geq 1
$$
for some $\gamma>0$. 
For all  $q\in\N$ and $\nu\geq 1$, we have the identity
$$
 m_\nu\cdot\eta\, =\, m_\nu\cdot(\eta-q\theta)+q\, m_\nu\cdot\theta,
$$
from which we obtain the inequalities
\begin{equation}\label{A}
\gamma<\|m_\nu\cdot\eta\|\leq n\max_{1\leq i\leq n}(|m_{\nu,i}|\cdot\|\eta_i-q\theta_i\|)+q\zeta_{\nu}.
\end{equation}
Since $\zeta_\nu$ is strictly decreasing and $\zeta_\nu\rightarrow 0$ as $\nu\rightarrow\infty$, there exists $\nu\geq 1$ such that
\begin{equation}\label{B}
\dfrac{\gamma}{2\zeta_\nu}\leq q\leq \dfrac{\gamma}{2\zeta_{\nu+1}}.
\end{equation}
On the one hand, from the inequalities \eqref{A} and the upper bound in \eqref{B}, we deduce that
\begin{equation}\label{E}
\max_{1\leq i\leq n}(\|\eta_i-q\theta_i\|\cdot|m_{\nu+1,i}|)>\dfrac{\gamma}{2n}.
\end{equation}
On the other hand, from the lower bound in \eqref{B} and the inequality \eqref{Minko}, it follows that                                   
$$
q\geq\dfrac{\gamma}{2} M_{\nu+1}.
$$
We deduce that
\begin{equation}\label{D}
q^{j_i}\geq c|m_{\nu+1,i}|\qquad \forall i=1,\ldots, n, 
\end{equation}
where 
$$
c=\min_{1\leq i\leq n}\Big(\Big(\dfrac{\gamma}{2}\Big)^{j_i}\Big).
$$
Finally, by combining \eqref{E} and \eqref{D}, we have that
$$
\max_{1\leq i\leq n}(\|\eta_i-q\theta_i\|q^{j_i})>\dfrac{\gamma c}{2n}.
$$
This concludes the proof of the proposition.
\end{proof}

We define the set
$$
\V:=\left\{\eta\in\R^n\, :\, \inf_{\nu\geq 1}\|m_\nu\cdot\eta\|>0\right\}.
$$
Clearly 
\begin{equation}\label{inclusion primera}
\V\subset\Tad.
\end{equation}

\section{The Cantor-type set $K(R)$}

In this section we construct the Cantor-type set $K(R)$ inside\\ $\Tad\cap\mathrm{Bad}(1,0,\ldots,0)
\cap\ldots\cap\mathrm{Bad}(0,\ldots,0,1)$. In order to lighten the notation, throughout this section we denote by $\M$ the set of best approximations in the sequence \eqref{sequence}, and for each $m\in\M$, by $M_m$ the quantity defined by \eqref{M}, i.e.
$$
M_m=\max_{1\leq i\leq n}(|m_{i}|^{1/j_i}).
$$
Hence 
$$
\V=\left\{\eta\in\R^n\, :\, \inf_{m\in\M}\|m\cdot\eta\|>0\right\}.
$$
We define the following partition of $\M$: 
\begin{equation}\label{Mk}
\Mk:=\left\{m\in\M : R^{k-1}\leq M_m< R^k\right\}\qquad (k\geq 0).
\end{equation}
Note that $\M_0=\emptyset$.
We have that $\M=\bigcup^\infty_{k=0}\Mk$.
\\

We also need, for each $1\leq i \leq n$, the following partitions of $\N$:
\begin{equation}\label{Qk}
\Qk^{(i)}:=\left\{q\in\N : R^{(k-1)j_i/2}\leq q< R^{kj_i/2}\right\}\qquad (k\geq 0).
\end{equation}
Note that $\mathcal{Q}_0^{(i)}=\emptyset$ and for each $1\leq i\leq n$, we have that $\N=\bigcup^\infty_{k=0}\Qk^{(i)}$.
\\

At the heart of the construction of $K(R)$ is constructing a collection $\F_k$ of hyperrectangles $H_{k}$ inside the hypercube $[0,1]^n$ that satisfy the following $n$ conditions: 
\begin{itemize}
\item[(0)] $|m\cdot\eta+p|\geq\epsilon\qquad \forall\eta\in H_k,\, \forall m\in\M_{k-1},\, \forall p\in\Z$;
\item[(1)] $q|q\eta_1-p|\geq\epsilon\qquad \forall\eta\in H_k,\, \forall q\in\mathcal{Q}_{k-1}^{(1)},\, \forall p\in\Z$;
\item[]
$\qquad\qquad\qquad\qquad\vdots$
\item[($n$)]  $q|q\eta_n-p|\geq\epsilon\qquad \forall\eta\in H_k,\, \forall q\in\mathcal{Q}_{k-1}^{(n)},\, \forall p\in\Z$
\end{itemize}
for some $\epsilon>0$.

We start by constructing a collection $(\G_k^{(0)})_{k\geq 0}$ of hyperrectangles satisfying condition (0). This construction is done by induction. Then we define a subcollection $\G^{(1)}_k\subset\G^{(0)}_k$ of hyperrectangles that also satisfy condition (1), a subcollection $\G^{(2)}_k\subset\G^{(1)}_k$ that also satisfies condition (2), etc. This process ends with a subcollection $\G^{(n)}_k$ that satisfies the $n$ conditions above. 
We would like to quantify $\#\G^{(n)}_k$. We can give a lower bound, but we cannot quantify the exact cardinal. So we refine the collection $\G_k^{(n)}$ by choosing a right and final subcollection $\F_k$ that we can quantify. 
\\

Let
$$
j_{\min}=\min_{1\leq i\leq n}(j_i),\qquad j_{\max}=\max_{1\leq i\leq n}(j_i).
$$
Let $R>4^{1/j_{\min}}$ and $\epsilon>0$ be such that
\begin{equation}\label{epsilon}
\epsilon<\dfrac{1}{2R^{2j_{\max}}}.
\end{equation}
The parameter $R$ will be chosen later to be sufficiently large in order to satisfy various conditions.

\subsection{The collection $\G^{(0)}_k$}

For each $m\in\M$ and $p\in\Z$, let
$$
\Delta(m,p):=\left\{x\in \R^n:\, |m\cdot x+p|<\epsilon\right\}.
$$
Geometrically, $\Delta(m,p)$ is the thickening of a hyperplane of the form
\begin{equation}\label{hyper}
\mL(m,p):=\left\{x\in\R^n:\, m\cdot x+p=0\right\}
\end{equation}
with width $2\epsilon/m_i$ in all the $x_i$-coordinate directions. 
\\

Next we describe the induction procedure in order to define the collection $(\G^{(0)}_k)_{k\geq 0}$.
We work within the closed hypercube $H_0=[0,1]^n$ and set $\G^{(0)}_0=\left\{H_0\right\}$. For $k\geq 0$, we divide each $H_{k}\in\G^{(0)}_{k}$ into new hyperrectangles $H_{k+1}$ of size 
$$
R^{-(k+1)j_1}\times\ldots\times R^{-(k+1)j_n}.
$$
Note that if $R^{j_i}\not\in\Z$ for some $1\leq i\leq n$, the division will not be exact, in the sense that the new hyperrectangles will not cover $H_k$. This division gives at least $\prod_{i=1}^n[R^{j_i}]> R-\sum_{i=1}^n R^{j_i}$ new hyperrectangles. Among these new hyperrectangles, we denote by $\G^{(0)}(H_{k})$ the collection of hyperrectangles $H_{k+1}\subset H_{k}$ satisfying
$$
H_{k+1}\cap\Delta(m,p)=\emptyset\qquad \forall m\in\Mk,\, \forall p\in\Z.
$$

We define 
$$
\G^{(0)}_{k+1}:=\bigcup_{H_{k}\in\G^{(0)}_{k}}\G^{(0)}(H_{k}).
$$
Hence $\G^{(0)}_{k+1}$ is nested in $\G^{(0)}_{k}$ and it is a collection of `good' hyperrectangles with respect to all the best approximations $m$ satisfying $M_m< R^k$ and all the integers $p$. The collection $\G^{(0)}(H_k)$ is the collection of `good' hyperrectangles that we obtain from the division of $H_k$.

Next we give a lower bound for $\#\G^{(0)}_k$. Actually,  for a fixed hyperrectangle $H_k\in\G^{(0)}_k$, we give a lower bound for the number of hyperrectangles $H_{k+1}\in \G^{(0)}(H_k)$. Alternatively, we give an upper bound for the number of `bad' hyperrectangles in $H_k$; these are the hyperrectangles $H_{k+1}\subset H_k$ that intersect the thickening $\Delta(m,p)$ of some hyperplane $\mL(m,p)$ with $m\in\Mk$.
Fact 1 and Fact 2 bound the number of thickenings $\Delta(m,p)$ with $m\in\M_k$ and $p\in\Z$ that intersect $H_k$.  Fact 3 bounds the number of hyperrectangles $H_{k+1}\subset H_k$ that are intersected by a thickening $\Delta(m,p)$ with $m\in\Mk$ and $p\in\Z$.
\\

\textbf{Fact 1.} We show that for each $k\geq 1$, the set $\Mk$ contains at most $2\cdot3^n(1+\log_2(R))$ best approximations.
Indeed, lemma \ref{lema lacunary} implies that 
\begin{align*}
M_{\nu+2\cdot3^n(1+\log_2(R))}&\geq 2^{1+\log_2(R)}M_\nu\\
&\stackrel{\eqref{Mk}}{\geq} 2^{1+\log_2(R)}R^{k-1}\\
&> R^k.
\end{align*}
Therefore, there are at most $2\cdot3^n(1+\log_2(R))$ best approximations in $\M_k$. 
\\

\textbf{Fact 2.} Fix $m\in\Mk$. We show that there are at most $2^{n}n$ thickenings $\Delta(m,p)$ that intersect $H_k$.
Indeed, suppose that two different thickenings $\Delta(m,p)$ and $\Delta(m,p')$ intersect the same edge of $H_k$. This edge of $H_k$ is a segment of a line which is parallel to an $x_l$-axis.
 Let $P=(y_1,\ldots,y_n)$ and $P'=(y_1',\ldots,y_n')$ denote the points of intersection of this line parallel to the $x_l$-axis with $\mL(m,p)$ and $\mL(m,p')$ respectively. 
The fact that $P$ and $P'$ respectively belong to $\mL(m,p)$ and $\mL(m,p')$ is described by the equations
\begin{equation}\label{P}
m\cdot y+p=0,\qquad
m\cdot y'+p'=0.
\end{equation}
The fact that $P$ and $P'$ both belong to a line parallel to the $x_l$-axis implies that $y_i=y_i'$ $\forall i\neq l$. Hence, by substracting the second equation in \eqref{P} to the first one, we have that
\begin{equation}\label{edge}
|y_l-y_l'|-\dfrac{2\epsilon}{|m_l|}\geq \dfrac{|p-p'|}{|m_l|}-\dfrac{2\epsilon}{|m_l|}> \dfrac{1}{R^{kj_l}}-\dfrac{1}{2R^{kj_l}}=\dfrac{1}{2}R^{-kj_l}.
\end{equation}
Since the length size of $H_k$ in the $x_l$-direction is $R^{-kj_l}$, the inequality \eqref{edge} implies that there are not more than two thickenings intersecting the same edge of $H_k$.
Thus the number of thickenings $\Delta(m,p)$ that intersect $H_k$ is at most twice the number of edges of $H_k$, and this is $2^{n}n$. 
\\

\textbf{Fact 3.} Given a thickening $\Delta(m,p)$, we give an upper bound for the number of hyperrectangles $H_{k+1}\subset H_k$ that intersect $\Delta(m,p)$. Fix $m\in\Mk$ and $p\in\Z$. Denote by $l$ the index such that $M_m=|m_l|^{1/j_l}$. 
Consider the projection of $\Delta(m,p)\cap H_k$ onto one of the faces of $H_k$ parallel to the plane given by the $x_l$-axis and an $x_i$-axis. We split this projected of $\Delta(m,p)\cap H_k$ into right triangles with perpendicular sides of length $2\epsilon/|m_l|$ and $2\epsilon/|m_i|$ respectively.
 From this splitting and the inequality
$$
\dfrac{2\epsilon}{|m_l|}<\dfrac{1}{2R^{j_l(k+1)}},
$$
we deduce that $\Delta(m,p)$ intersects at most $2[R^{1-j_{\min}}]$ hyperrectangles $H_{k+1}\subset H_k$.
\\
 
\textbf{Conclusion.} There are at most $[2^{n+2}3^n n(1+\log_2(R))R^{1-j_{\min}}]$ hyperrectangles $H_{k+1}\subset H_k$ that intersect some $\Delta(m,p)$ with $m\in\M_k$, $p\in\Z$. Hence 
$$
\#\G^{(0)}(H_k)\geq R-\sum_{i=1}^n R^{j_i}-[2^{n+2} 3^n n(1+\log_2(R))R^{1-j_{\min}}].
$$

\subsection{The subcollections $\G_k^{(i)}$}

For each $q\in\N$ and $p\in\Z$, consider the sets
\begin{equation}\label{thicke 2}
\Gamma_i(q,p):=\left\{x\in \R^n:\, q|qx_i-p|<\epsilon\right\}\qquad (1\leq i\leq n).
\end{equation}
Geometrically, each $\Gamma_i(q,p)$ is a thickening of a hyperplane described by the equation $x_i=p/q$ with width $2\epsilon/q^2$ in the $x_i$-coordinate direction.
\\

We construct a tower of subcollections 
$$
\G_k^{(n)}\subset\G_k^{(n-1)}\subset \ldots\subset\G_k^{(1)}\subset\G_k^{(0)},
$$
where each $\G_k^{(i)}$ consists of hyperrectangles in $\G_k^{(i-1)}$ which points avoid each thickening $\Gamma_i(q,p)$ for $q\in\mathcal{Q}_{k-1}^{(i)}$.
More precisely,
for $1\leq i\leq n$, we form $\G_k^{(i)}$ by letting
$$
\G^{(i)}(H_k):=\left\{H_{k+1}\in\G^{(i-1)}(H_k):\, H_{k+1}\cap\Gamma_i(q,p)=\emptyset\,\, \forall q\in\Qk^{(i)}\right\}
$$
and 
$$
\G^{(i)}_{k+1}:=\bigcup_{H_k\in\G_k^{(i-1)}} \G^{(i)}(H_k).
$$
Clearly the hyperrectangles in $\G^{(i)}_{k+1}$ satisfy the conditions (0),(1),...,($i$), so the collection $\G^{(n)}_{k}$ satisfies the $n$ conditions (0),...,($n$). 
\\

Next, for each $1\leq i\leq n$ and $H_k\in\G_k^{(i-1)}$, we give a lower bound of $\#\G^{(i)}(H_k)$.
Suppose that there are two pairs $(q,p)$ and $(q',p')$ in $\Qk^{(i)}\times\Z$ such that 
$$
H_{k}\cap\Gamma_i(q,p)\neq\emptyset,\qquad H_{k}\cap\Gamma_i(q',p')\neq\emptyset.
$$
In other words, suppose there exist $\eta,\eta'$ in $H_k$ such that
\begin{equation}\label{ineq eta}
q|q\eta_i-p|<\epsilon,\qquad q'|q'\eta'_i-p'|<\epsilon.
\end{equation}
Then, by \eqref{Qk} and \eqref{epsilon}, we have 
\begin{equation}\label{pq}
\left|\dfrac{p}{q}-\dfrac{p'}{q'}\right|-\dfrac{\epsilon}{q^2}-\dfrac{\epsilon}{q'^2}\geq \dfrac{1}{qq'}-\dfrac{\epsilon}{q^2}-\dfrac{\epsilon}{q'^2}> \dfrac{1}{R^{kj_i}}-\dfrac{1}{2R^{kj_i}}=\dfrac{1}{2}R^{-kj_i}.
\end{equation}
Since the length sides of $H_k$ in the $x_i$-direction is $R^{-kj_i}$, the inequality \eqref{pq} implies that at most two thickenings of the form \eqref{thicke 2} can intersect $H_k$.
\\

Now, from  \eqref{Qk} and \eqref{epsilon}, it follows that if $\eta\in\Gamma_i(q,p)$, then
$$
\left|\eta_i-\dfrac{p}{q}\right|<\dfrac{\epsilon} {q^2}< \dfrac{1}{2}R^{-kj_i},
$$
which implies that each thickening $\Gamma_i(q,p)$ intersects at most 
$$
2[R^{j_1}]\times\ldots\times\widehat{[R^{j_i}]}\times\ldots\times[R^{j_n}]\leq 2 [R^{1-j_i}]
$$
hyperrectangles $H_{k+1}\subset H_k$.
\\

Therefore, there are at most $4 [R^{1-j_{\min}}]$ hyperrectangles  $H_{k+1}\subset H_k$ that do not satisfy condition (i). Hence
\begin{equation}\label{cardinal segundo}
\#\G^{(i)}(H_k)\geq R-\sum_{i=1}^n R^{j_i}-[2^{n+2} 3^n n(1+\log_2(R))R^{1-j_{\min}}]-4i[R^{1-j_{\min}}].
\end{equation}

\subsection{The right subcollection $\F_k$}

We choose a subcollection of $\G^{(n)}_k$ that we can exactly quantify in the following way. Let $\F_0:=\G^{(0)}_0$. Choose $R$ sufficiently large so that $[R-\sum_{i=1}^n R^{j_i}-2^{n+2}3^n n(1+\log_2(R))\cdot R^{1-j_{\min}}-4nR^{1-j_{\min}}]> 1$.
 For $k\geq 0$,
for each $H_k\in\F_k$, we choose exactly $[R-\sum_{i=1}^n R^{j_i}-2^{n+2}3^n n(1+\log_2(R))R^{1-j_{\min}}-4nR^{1-j_{\min}}]$ hyperrectangles from the collection $\G^{(n)}(H_k)$ and denote this collection by $\F(H_k)$. Trivially, 
\begin{equation}\label{cardinal exacto}
\#\F(H_k)=[R-\sum_{i=1}^n R^{j_i}-2^{n+2} 3^n n(1+\log_2(R))R^{1-j_{\min}}-4nR^{1-j_{\min}}]>1,
\end{equation}
so each hyperrectangle $H_k\in\F_k$ gives rise to exactly the same number of hyperrectangles $H_{k+1}$ in $\F(H_k)$. 
Finally, define
$$
\F_{k+1}:=\bigcup_{H_k\in\F_k}\F(H_k).
$$
This completes the construction of the Cantor-type set
$$
K(R):=\bigcap_{k=0}^\infty\F_k.
$$
By construction, we have $K(R)\subset\V\cap\mathrm{Bad}(1,0\ldots,0)\cap\ldots\cap\mathrm{Bad}(0,\ldots,0,1)$. Moreover, in view of \eqref{cardinal exacto}, we have
\begin{align}\label{cardinal teorico}
\#\F_{k+1}&=\#\F_{k}\, \#\F(H_k)\\
&=[R-\sum_{i=1}^n R^{j_i}-2^{n+2} 3^n n(1+\log_2(R))R^{1-j_{\min}}-4nR^{1-j_{\min}}]^{k+1}.
\end{align}

\section{The measure $\mu$ on $K(R)$}

We now describe a probablity measure $\mu$ supported on the Cantor-type set $K(R)$ constructed in the previous section. The measure we define is analogous to the probability measure used in \cite{PV} and \cite{BPV} on a Cantor-type set of $\R^2$. For any hyperrectangle $H_k\in\F_k$ we attach a weight $\mu(H_k)$ which is defined recursively as follows: for $k=0$,
$$
\mu(H_0)=\dfrac{1}{\#\F_0}=1
$$
and for $k\geq 1$,
$$
\mu(H_k)=\dfrac{1}{\#\F(H_{k-1})}\mu(H_{k-1})\qquad (H_k\in\F(H_{k-1})).
$$
This procedure defines inductively a mass on any hyperrectangle used in the construction of $K(R)$. Moreover, $\mu$ can be further extended to all Borel subsets $X$ of $\R^n$, so that $\mu$ actually defines a measure supported on $K(R)$, by letting
$$
\mu(X)=\inf\sum_{H\in\mathcal{C}}\mu(H)
$$
where the infimum is taken over all coverings $\mathcal{C}$ of $X$ by rectangles $H\in\left\{\F_k\, :\, k\geq 0\right\}$. For further details, see \cite{F}, Proposition 1.7.

Notice that, in view of \eqref{cardinal teorico}, we have
$$
\mu(H_k)=\dfrac{1}{\#\F_k}\qquad (k\geq 0).
$$

A classical method for obtaining a lower bound for the Hausdorff dimension of an arbitrary set is the following mass distribution principle (see \cite{F} p. 55).

\begin{lemma}[mass distribution principle]
Let $\delta$ be a probability measure supported on a subset $X$ of $\R^n$. Suppose there are positive constants $c,s$ and $l_0$ such that
\begin{equation}\label{mdp}
\delta(S)\leq cl^s
\end{equation}
for any hypercube $S\subset\R^n$ with side length $l\leq l_0$. Then $\dim(X)\geq s$.
\end{lemma}

The goal in the next section is to prove that there exist constants $c$ and $l_0$ satisfying \eqref{mdp} with $\delta=\mu$, $X=K(R)$ and $s=n-\lambda(R)$, where $\lambda(R)\rightarrow 0$ as $R\rightarrow\infty$. Then from the mass distribution principle it will follow that $\dim(K(R))=n$.

\section{A lower bound for $\dim(K(R))$}

Recall that 
$$
j_{\min}=\min_{1\leq i\leq n}(j_i).
$$
Let $k_0$ be a positive integer such that 
\begin{equation}\label{k0}
R^{-kj_i}< R^{-(k+1)j_{\min}}\qquad\forall j_i\neq j_{\min}\mbox{ and } k\geq k_0.
\end{equation}

Consider an arbitrary hypercube $S$ of side length $l\leq l_0$ where $l_0$ satisfies
\begin{equation}\label{l0}
l_0<R^{-(k_0+1)j_{\min}}
\end{equation}
together with a second inequality to be determined later. 
We can choose $k > k_0$ so that
\begin{equation}\label{l}
R^{-(k+1)j_{\min}}< l < R^{-kj_{\min}}. 
\end{equation}
From the inequality \eqref{k0} it follows that
\begin{equation}\label{lcounting}
l>R^{-kj_i}\qquad\forall j_i\neq j_{\min}.
\end{equation}
Then it is easy to see that $S$ intersects at most $2^nl^{n-1}\prod_{j_i\neq j_{\min}} R^{kj_i}$ hyperrectangles $H_k\in\F_k$, so
$$
\mu(S)\leq 2^nl^{n-1}\prod_{j_i\neq j_{\min}} R^{kj_i}\mu(H_k)=2^nl^{n-1} R^{k-kj_{\min}}\dfrac{1}{\#\F_k}.
$$
Since $R^{(k+1)j_{\min}}> l^{-1}$ (see \eqref{l}), we have that
$$
\mu(S) \leq 2^n l^n R^{j_{\min}}R^k\dfrac{1}{\#\F_k}.
$$
Remember that we mentioned in Section 3 that later we would choose the parameter $R$ big enough so that it satisfies various conditions. We choose $R$ so that 
$$
R^{-1}\sum_{i=1}^n R^{j_i}-2^{n+2} 3^n n (1+\log_2(R))R^{-j_{\min}}-4nR^{-j_{\min}}-R^{-1}\leq 2^{-1}.
$$
Then, by \eqref{cardinal teorico} we have that
$$
\mu(S)\leq 2^n l^n R^{j_{\min}} 2^k.
$$
We choose 
$$
k\geq\log(R)\quad\mbox{and}\quad
\lambda(R)=\dfrac{1+\log(2)}{j_{\min}\log(R)},
$$
so
$$
\mu(S)\leq 2^n l^n R^{kj_{\min}\lambda(R)}.
$$
Since $R^{kj_{\min}} < l^{-1}$ (see \eqref{l}), it follows that
$$
\mu(S)\leq 2^n l^{n-\lambda(R)}.
$$ 
Finally, by applying the mass distribution principle we obtain
$$
\dim K(R)\geq n-\lambda(R)\rightarrow n\qquad \mbox{as } R\rightarrow\infty.
$$

\end{document}